\documentclass{amsart}
\usepackage{amsmath,amsfonts,amsthm,amssymb,setspace}
\usepackage{tikz}
\usetikzlibrary{arrows,decorations.pathmorphing,backgrounds,positioning,fit}
\newtheorem{thm}{Theorem}[subsection]

\newtheorem{lemma}{Lemma}[subsection]

\newtheorem{coro}{Corollary}[subsection]

\theoremstyle{definition}
\newtheorem{defn}{Definition}[subsection]
\theoremstyle{remark}
\newtheorem{rem}{Remark}[subsection]
\numberwithin{equation}{section}
\def\O{\mathcal{O}}
\def\A{\mathcal{C}}
\def\X{\mathcal{X}}

\newcommand{\gmat}[2][ccccccccccccccccccccccccccccccccc]{\left[\begin{array}{#1} #2\\ \end{array}\right]}

\title{The Weil-Petersson 2-form on an acyclic cluster variety.}

\author{Greg Muller}
\address{Department of Mathematics,
Louisiana State University, Baton Rouge, LA 70808, USA}
\email{gmuller@lsu.edu}

\begin{document}

\begin{abstract}
The Weil-Petersson 2-form on a cluster variety is a 2-form on a certain open smooth subvariety; the union of the cluster tori.  We show that for acyclic cluster varieties, the Weil-Petersson 2-form extends to a regular K\"ahler 2-form on the entire cluster variety.
\end{abstract}

\maketitle

\section{Introduction}

Cluster algebras are commutative algebras generated by a distinguished set of elements which may be recursively constructed from some initial data.  They were introduce by Fomin and Zelevinsky \cite{FZ02} in the study of total positivity and canonical bases.  The corresponding \emph{cluster varieties} have been shown to occur both as double Bruhat cells in semisimple Lie groups \cite{BFZ05} and as complexifications of Penner's decorated Teichm\"uller spaces \cite{GSV05}, \cite{FG06}.  

Drawing inspiration from the Weil-Petersson 2-form on a decorated Teichm\"uller space, Gekhtman, Shapiro and Vainshtein \cite{GSV05} defined a 2-form $\omega$ on a smooth, dense subset of a cluster variety.  This form (also called the Weil-Petersson form) is a geometric manifestation of some of the extra structure on the cluster algebra beyond just being a commutative algebra.  

It is natural to ask whether the form $\omega$ extends to the entire cluster variety.  However, cluster varieties may be singular, and the geometric definition of a differential form breaks down at a singularity.  Thus it is necessary to generalize to the language of \emph{K\"ahler forms}, which are an algebraic analog of differentials forms which may be defined at singularities.

In this paper, we show that, for a large class of cluster varieties called \emph{acyclic cluster varieties}, the Weil-Petersson form $\omega$ extends as a K\"ahler 2-form to the entirety of the cluster variety (Theorem \ref{thm: Main}).  As a consequence, when the cluster variety is smooth, the form $\omega$ extends as a (differential) 2-form (Corollary \ref{coro: Main}).

The paper concludes with several acyclic examples, both singular and smooth, as well as a non-acyclic counterexample which shows that the Weil-Petersson form $\omega$ does not extend for all cluster varieties.

\section{Cluster Algebras}

We begin by briskly reviewing the definition of a cluster algebra (of geometric type).  The reader looking for a more useful introduction is directed to \cite{GSV10}.

\subsection{Cluster algebras}  A rank $n$ cluster algebra is a commutative algebra with a canonical set of generators (\emph{cluster variables}) such that the complete set of generators may be recursively constructed from special $n$-element subsets (\emph{clusters}).  The input data is an integral $m\times n$ matrix $A$ which is \emph{skew-symmetrizable}; that is, there is an integral diagonal $m\times m$ matrix $D$ such that the restriction of $DA$ to the first $m$ columns is skew-symmetric.  

Given a pair (called a \emph{seed}) of an $m\times n$ skew-symmetrizable matrix $B$ and a set of variables $\mathbf{f}=\{f_1,f_2,...,f_n\}$, the \emph{mutation at $k\in \{1,...,m\}$} of $(B,\mathbf{f})$ is the new seed defined by
\begin{itemize}
\item For $1\leq i\leq m, 1\leq j\leq n$, set 
\[ \mu_k(B)_{ij} := \left\{\begin{array}{cc}
-B_{ij} & \text{if } k=i\text{ or }j\\
B_{ij} + \frac{|B_{ik}|B_{kj}+B_{ik}|B_{kj}|}{2} & \text{if } k\neq i,j
\end{array}\right\}\]
\item For $i\neq k$, set $\mu_k(f_i):=f_i$, and set
\[ \mu_k(f_k):=
\left(\prod_{j,\;B_{ij}>0} f_j^{B_{ij}}+\prod_{j,\;B_{ij}<0} f_j^{-B_{ij}}\right)f_i^{-1}\]
\end{itemize}

The new matrix $\mu_k(B)$ is also skew-symmetrizable, and so the procedure of mutation may be iterated at any sequence in the indices $\{1,...,m\}$.

Any seed $(B,\mathbf{f})$ obtained from $(A,\mathbf{x}=\{x_1,...x_n\})$ by a sequence of mutations is called \emph{equivalent} to $(A,\mathbf{x})$.  The \textbf{clusters} are the sets $\mathbf{f}=\{f_1,...,f_n\}$ equivalent to $\mathbf{x}$, and the \textbf{cluster variables} are elements of the union of the clusters.  Note that, for $m<i\leq n$, the element $x_i$ is never mutated and so $f_i=x_i$; these are the \textbf{frozen variables} and they are in every cluster.

\begin{defn}The (complex) \textbf{cluster algebra} $\A(A)$ associated to $A$ is the $\mathbb{C}$-subalgebra of $\mathbb{C}(x_1,x_2,...,x_n)$ generated by the set of all cluster variables, together with the inverses of the frozen variables, $x_{m+1}^{-1},...,x_n^{-1}$.
\end{defn}

\begin{rem}
We require an algebraically closed field for geometric applications, though any algebraically closed field will suffice.  
\end{rem}

For a fixed seed $(B,\mathbf{f})$ and cluster variable $f_i\in \mathbf{f}$, we use the somewhat vague notation $f_i':=\mu_i(f_i)$.  Then it is immediate that the following relation holds in the cluster algebra.
\begin{equation}\label{eqn: Cluster}
f_if_i' = \prod_{j,\;B_{ij}>0} f_j^{B_{ij}}+\prod_{j,\;B_{ij}<0} f_j^{-B_{ij}}
\end{equation}

A fundamental lemma is the Laurent Phenomenon.  For $\mathbf{f}$ a cluster, let $\mathbb{C}[\mathbf{f}^{\pm1}]:= \mathbb{C}[f_1^{\pm1},f_2^{\pm1},...,f_n^{\pm1}]$, the algebra of Laurent polynomials in $\mathbf{f}$.
\begin{lemma}\cite[Theorem 3.1]{FZ02} For any cluster $\mathbf{f}$ in $\A(A)$, we have $\A(A)\subset \mathbb{C}[\mathbf{f}^{\pm1}]$.
\end{lemma}

\subsection{Acyclic cluster algebras}

While the structure of a general cluster algebra may be very mysterious, there is an explicit structure theorem in the case of \emph{acyclic} cluster algebras.

Given a skew-symmetrizable matrix $A$, a \emph{directed cycle} is a sequence of indices $i_1,i_2,...i_j\in \{1,...,m\}$ such that $i_j=i_1$ and $A_{i_ki_{k+1}}>0$ for all $1\leq k<j$.  A skew-symmetrizable matrix is \emph{acyclic} if it has no directed cycles.  By a standard abuse of notation, let $x_i':=\mu_i(x_i)$ be the mutation from the initial cluster $\mathbf{x}$, and let $P_i:=x_ix_i'$, which is a binomial in the variables $\{x_1,...,x_n\}$.

\begin{lemma}\cite[Corollary 1.21]{BFZ05}\label{lemma: AcyclicStr}
Let $A$ be an acyclic skew-symmetrizable matrix.  Then 
\[\A(A)=\mathbb{C}[x_1,x_2,...x_m,x_{m+1}^{\pm1},...,x_n^{\pm1}, x_1',x_2',...x_m']/ \langle x_1x_1'-P_1,...x_mx_m'-P_m\rangle\] 
\end{lemma}
The main consequence for our purposes is that $\A(A)$ is finitely-generated.

\begin{defn}\cite{BFZ05}
For a skew-symmetrizable matrix $A$, the associated cluster algebra $\A(A)$ is \textbf{acyclic} if there is some acyclic matrix $B$ equivalent to $A$.
\end{defn}

\section{Cluster Varieties}

\subsection{Cluster varieties}

For $A$ acyclic, the algebra $\A(A)$ is a finitely-generated integral domain.  Hence, $\A(A)$ is the coordinate ring of the affine subvariety $\X(A)$ of $\mathbb{C}^{2n}$ defined by the equations
\[ x_ix_i'=P_i,\;\;\;(1\leq i\leq m);\;\;\; x_ix_i^{-1}=1,\;\;\; (m<i\leq n)\]
Different choices of acyclic cluster in $\A(A)$ will determine different embeddings of the same affine variety $\X(A)$, which can be identified with the set of maximal ideals in the algebra $\A(A)$.

\begin{defn}
The affine variety $\X(A)$ is the \textbf{cluster variety} of $A$.
\end{defn}

\begin{rem}
If $\A(A)$ is not acyclic, the `cluster variety of $\A(A)$' may still be defined as the space of maximal ideals.  However, if $\A(A)$ is not finitely-generated, then this does not embed inside any finite-dimensional affine space, and so it fails to be a true variety.
\end{rem}


Given a cluster $\mathbf{f}$, the inclusion 
\[\A(A)\hookrightarrow \mathbb{C}[\mathbf{f}^{\pm1}]\]
is dual to an inclusion of varieties
\[ (\mathbb{C}^*)^n\hookrightarrow \X(A)\]
The image of this map is the \textbf{cluster torus} $\mathbb{T}_{\mathbf{f}}$ associated to $\mathbf{f}$.  Concretely, $\mathbb{T}_{\mathbf{f}}$ is the set of points in $\X(A)$ where the coordinate functions $f_1,f_2,...,f_n$ are non-zero.  At those points, the value of the other elements of $\A(A)$ may be determined by their inclusion into $\mathbb{C}[\mathbf{f}^{\pm1}]$; the requirement that the $f_i$ are non-zero is necessary for the denominators of these Laurent polynomials not to vanish.

The union of all the cluster tori is a subvariety denoted $\X'(A)$.\footnote{It may not be an affine variety; it is only \emph{quasi-affine}.}  As a union of smooth varieties, $\X'(A)$ is again smooth.  For this reason, it is called the \textbf{cluster manifold}.  However, this name is somewhat misleading, since there may be other smooth points in $\X(A)$.

\subsection{The deep variety} The complement $\X_{d}(A):=\X(A) - \X'(A)$ is called the \textbf{deep variety}, and it can be smooth, singular or empty.  Even for acyclic cluster varieties, the structure of the deep variety is poorly understood.  Many cluster variables must vanish at deep points, and so it is worth understanding what subsets of cluster variables can simultaneously vanish.

\begin{lemma}
Let $p\in \X(A)$, and $(B,\mathbf{f})$ be a seed equivalent to $(A,\mathbf{x})$.  If $f_a, f_b\in \mathbf{f}$ with $B_{ab}\neq 0$ and $f_a(p)=f_b(p)=0$, then there is a directed cycle $(f_{k_1},f_{k_2},...f_{k_\alpha})$ of cluster variables in $\mathbf{f}$ such that 
\[f_{k_1}(p)=f_{k_2}(p)=...=f_{k_\alpha}(p)=0\]
\end{lemma}
\begin{proof}
Possibly switching the labels $a$ and $b$, we may assume that $B_{ab}>0$; set $f_{k_1}=f_a$ and $f_{k_2}=f_b$.  Let $f_b'$ denote the mutation of $f_b$ in $\mathbf{f}$.  Equation \eqref{eqn: Cluster} when evaluated at $p$.
\[f_b(p)f_b'(p)=\prod_{j,\;B_{bj}>0} [f_j(p)]^{B_{bj}}+\prod_{j,\;B_{bj}<0} [f_j(p)]^{-B_{bj}}\]
Since $f_b(p)=0$, the left side vanishes.  Since $f_a(p)=0$ and $B_{ba}=-B_{ab}< 0$, the second term on the right vanishes.  Hence, there is some $f_{c}$ such that $B_{bc}>0$ and $f_c(p)=0$; set $f_{k_3}=f_c$.

Iterating this procedure gives a list $f_{k_1},f_{k_2},...$.  Since there are a finite number of vertices, eventually this sequence must wrap around, producing the desired oriented cycle.
\end{proof}

\begin{coro}\label{coro: NoAdj}
Let $p\in \X(A)$, and $(B,\mathbf{f})$ be an \emph{acyclic} seed equivalent to $(A,\mathbf{x})$. If $f_a, f_b\in \mathbf{f}$ with $f_a(p)=f_b(p)=0$, then $B_{ab}= 0$.
\end{coro}

\section{K\"ahler Forms}

We review the notion of a K\"ahler form (also called a K\"ahler differential).  This is an algebraic generalization of the notion of a differential form (with regular coefficients) which can be defined at singularites.  

\subsection{The K\"ahler forms of an algebra}

Let $R$ be a commutative $\mathbb{C}$-algebra.  The $R$-module of \textbf{K\"ahler forms} $\Omega(R)$ is the $R$-module generated by symbols of the form $\{dr|\,r\in R\}$, with relations given by
\[ d(\lambda r)=\lambda(dr), \;\;\;  d(r+s) = dr+ds,\;\;\;  d(rs) = r(ds)+s(dr);\;\;\; \lambda\in \mathbb{C}, \;r,s\in R\]
There is a natural linear map $d:R\dashrightarrow \Omega$ which sends $r$ to $dr$.  This map is \underline{not} an $R$-module map, but it is a \emph{derivation}.\footnote{In fact, $d$ is the \emph{universal} derivation, in the sense that any other derivation $\delta:R\dashrightarrow M$ can be uniquely factored through $d$.}  Define the $R$-module of (regular) \textbf{K\"ahler $i$-forms} by
\[\Omega^i(R):=\wedge^i_R\Omega(R),\]
the $i$th exterior power of $\Omega(R)$.  

These constructions commute with localization; that is, if $S$ is a localization of $R$, then $\Omega(S)=S\otimes_R\Omega(R)$ and $\Omega^i(S)=S\otimes_R \Omega^i(R)$.  If $S$ is the fraction field of $R$, the module $\Omega^i(S)=S\otimes_R\Omega^i(R)$ is called the module of \textbf{rational K\"ahler $i$-forms}.

\subsection{K\"ahler forms in algebraic geometry}

When $R$ is the coordinate ring $\O_X$ of a smooth affine variety $X$, these modules have a geometric meaning.  Let $\mathcal{T}_X$ denote the $\O_X$-module of regular vector fields on $X$, and let $\mathcal{T}^*_X=Hom_X(\mathcal{T}_X,\O_X)$ denote the $\O_X$-module of regular 1-forms on $X$.  Then, the K\"ahler $i$-forms can be canonically identified with the $R$-module of regular $i$-forms.  
\begin{lemma}\label{lemma: KtoD}
Let $\O_X$ be the coordinate ring of a smooth affine variety $X$.  Then $\Omega(\O_X)=\mathcal{T}^*_X$, and so $\Omega^i(\O_X)=\wedge^i_X\mathcal{T}^*_X$.
\end{lemma}

Thus, K\"ahler $i$-forms coincide with regular $i$-forms on smooth varieties.  However, K\"ahler $i$-forms can be defined for singular varieties, when the notion of an $i$-form breaks down in the absence of a well-defined tangent space.

If $X$ is an affine variety, with a open affine subvariety $Y$, then the localization map
\[ \Omega^i(\O_X)\rightarrow \O_Y\otimes_X\Omega^i(O_X)=\Omega^i(\O_Y)\]
allows K\"ahler $i$-forms on $X$ to be restricted to K\"ahler $i$-forms on $Y$.  Combined with the above lemma, we see that K\"ahler $i$-forms on $X$ can be restricted to $i$-forms on $Y$ when $Y$ is smooth.
\begin{coro}
Let $X$ be an affine variety, and let $Y$ be a smooth open affine subvariety.  Then there is a natural restriction map $\Omega^i(\O_X)\rightarrow \wedge^i_Y\mathcal{T}^*_Y$.
\end{coro}
Therefore, even for singular varieties, the K\"ahler forms are differential forms on the smooth locus, with a formal algebraic extension to the singularities.

\section{The Weil-Petersson Form.}
From now on, a fixed skew-symmetrizable matrix $A$ will be chosen, and suppressed from notation.  We define the Weil-Petersson form introduced in \cite{GSV05}, a 2-form $\omega$ defined on the cluster manifold $\X'$.  Algebraically, this 2-form may be interpreted as a rational K\"ahler 2-form $\omega$ on the cluster algebra $\A$.  We then show that, in the acyclic case, this $\omega$ is a (regular) K\"ahler 2-form on $\A$. 

\subsection{The Definition of $\omega$.}

Let $(B, \mathbf{f})$ be a seed equivalent to $(A,\mathbf{x})$.  Define a rational K\"ahler 2-form $\omega_{\mathbf{f}}$ on the cluster algebra $\A$ by
\[ \omega_{\mathbf{f}} :=\sum_{i,j}\frac{B_{ij}}{f_if_j}df_i\wedge df_j\]
Since the denominator is a monomial in $\mathbf{f}$, this is a K\"ahler 2-form in $\mathbb{C}[\mathbf{f}^{\pm1}]\otimes_{\A}\Omega(\A)$.

Let $\mathbb{T}_{\mathbf{f}}$ be the cluster tori associated to $\mathbf{f}$, whose coordinate ring is naturally $\mathbb{C}[\mathbf{f}^{\pm1}]$.  Because $\mathbb{T}_{\mathbf{f}}$ is smooth, by Lemma \ref{lemma: KtoD}, $\omega_{\mathbf{f}}$ is equivalent to a regular 2-form on $\mathbb{T}_{\mathbf{f}}$.  

\begin{lemma}\cite[Theorem 2.1]{GSV05} For two clusters $\mathbf{f}$ and $\mathbf{f}'$, the corresponding rational K\"ahler 2-forms $\omega_{\mathbf{f}}$ and $\omega_{\mathbf{f}'}$ agree.  Hence, $\omega:=\omega_{\mathbf{f}}$ is a well-defined rational K\"ahler 2-form on $\A$ independant of a choice of cluster.
\end{lemma}

This form $\omega$ is called the \textbf{Weil-Petersson form} on $\X$.  The 2-form $\omega$ was originally defined as a differential 2-form on the smooth subvariety $\X'\subset \X$, but it can be regarded as a rational K\"ahler 2-form immediately from its construction.

\subsection{Regularity of $\omega$.}

This paper seeks to understand the extension of $\omega$ to the deep variety $\X_d=\X-\X'$.  Passing from differential 2-forms to K\"ahler 2-forms is necessary because $\X_d$ may contain singularities.  
We will show that $\omega$ is regular in the acyclic case; that is, that the equations defining it do not blow up anywhere on $\X_d$.  

First, a lemma which is more broadly applicable than just the acyclic case.

\begin{lemma}\label{lemma: Main}
Let $\A$ be a cluster algebra, and let $p\in \X$. If there exists a seed $(B,\mathbf{f})$ where $f_i(p)=f_j(p)=0$ implies that $B_{ij}=0$, then $\omega$ is locally regular at $p$.
\end{lemma}
\begin{proof}
Let $f_i\in \mathbf{f}$ be such that $f_i(p)=0$, and let $f_i'$ denote the mutation of $f_i$ in $\mathbf{f}$.  Then apply $d$ to Equation \eqref{eqn: Cluster} to get
\[ f_idf_i'+f_i'df_i =  \left(\sum_{j, B_{ij}>0}\frac{B_{ij}df_j}{f_j}\right)\prod_{j, B_{ij}>0}f_j^{B_{ij}}+\left(\sum_{j, B_{ij}<0}\frac{-B_{ij}df_j}{f_j}\right)\prod_{j,B_{ij}<0}f_j^{-B_{ij}}\]
Wedge both side with $df_i$ and divide by $f_i$ to get
\[ df_i\wedge df_i' = \left(\sum_{j, B_{ij}>0}B_{ij}\frac{df_i\wedge df_j}{f_if_j}\right)\prod_{j, B_{ij}>0}f_j^{B_{ij}}-\left(\sum_{j, B_{ij}<0}B_{ij}\frac{df_i\wedge df_j}{f_if_j}\right)\prod_{j,B_{ij}<0}f_j^{-B_{ij}}\]
\[= \left( \sum_{ j} B_{ij}\frac{df_i\wedge df_j}{f_if_j}\right)\left[\prod_{j, B_{ij}>0} f_j^{B_{ij}}\right]-\left( \sum_{j,B_{ij}<0} B_{ij}\frac{df_i\wedge df_j}{f_j}\right)f_i'\]
\[ \sum_{j}B_{ij} \frac{df_i\wedge df_j}{f_if_j} = \left[\prod_{j,B_{ij}>0}f_j^{-B_{ij}}\right]\left(df_i\wedge df_i'+
f_i'\left( \sum_{j,B_{ij}<0} B_{ij}\frac{df_i\wedge df_j}{f_j}\right)\right)\]
Notice that a cluster variable $f_j$ only appears in the denominator of the right-hand-side if $B_{ij}\neq0$, and so by assumption $f_j(p)\neq 0$.  Therefore, the expression on the right-hand-side is a locally regular K\"ahler 2-form at $p$.

The Weil-Petersson form $\omega_{\mathbf{f}}$ written in terms of the cluster $\mathbf{f}$ can have its terms grouped into sums of the above form, plus other terms whose denominator does not vanish at $p$.  Therefore, the Weil-Petersson form $\omega$ is locally regular at $p$. 
\end{proof}

\begin{thm}\label{thm: Main}
Let $\A$ be an acyclic cluster algebra.  Then $\omega$ is a regular K\"ahler 2-form.
\end{thm}
\begin{proof}
Fix an acyclic seed $(B,\mathbf{f})$, and let $p\in \X$. By Corollary \ref{coro: NoAdj}, the hypothesis of Lemma \ref{lemma: Main} holds, and so $\omega$ is locally regular at $p$. Since this holds for all $p\in \X$, $\omega$ is regular on the entirety of $\X$.
\end{proof}
\begin{rem}
The proof of the lemma gives an explicit algorithm for expressing $\omega$ as a locally regular K\"ahler 2-form at a specific point $p$, but it does not give a globally regular expression (ie, denominator free) for $\omega$, though the theorem guarantees such an expression must exist.
\end{rem}

\begin{coro}\label{coro: Main}
If $\X$ is a smooth acyclic cluster variety, then $\omega$ extends to a regular 2-form on the entirety of $\X$.
\end{coro}

\section{Examples}  We close with several examples, to illustrate the structure of $\omega$ at the deep points in some small cluster varieties, and to show that $\omega$ may not be regular when the cluster algebra is not acyclic.

\subsection{The open double Bruhat cell in $SL_2(\mathbb{C})$.} Let 
\[ A = \gmat{ 0 & 1 & 1}\]
Denote the initial cluster by $\{x, c_1,c_2\}$, where the frozen variables are denoted by $c_i$.  Then the remaining cluster variable is $x'=\frac{c_1c_2+1}{x}$.  The cluster algebra $\A$ is given by
\[ \A(A) = \mathbb{C}[x,x',c_1^{\pm1},c_2^{\pm1}]/\langle xx'-(c_1c_2+1)\rangle\]
By the map
\[ (x,x',c_1,c_2)\rightarrow \gmat{ x & c_1 \\ c_2 & x' },\]
the cluster variety $\X(A)$ may be identified with the set of $2\times 2$ matrices of determinant $1$ such that $c_1\neq 0$ and $c_2\neq 0$.  This is the \emph{open double Bruhat cell} in $SL_2(\mathbb{C})$, and a motivating example of a cluster variety.

The union of the cluster tori $\X'(A)$ is the set of points where $x$ or $x'$ is non-zero, and so the deep variety $\X_d(A)$ is the set where $x=x'=0$; it consists of matrices of the form
\[ \gmat{0 & -c_1 \\ c_1^{-1} & 0} \]
As an open subvariety of the smooth variety $SL_2(\mathbb{C})$, the cluster variety $\X(A)$ is smooth even at $\X_d(A)$.

The Weil-Petersson form can be written
\[\omega = \frac{dx\wedge dc_1}{xc_1}+\frac{dx\wedge dc_2}{xc_2}=\frac{dx\wedge dx'}{c_1c_2}\]
In the second expression, it is clear it is globally regular on $\X(A)$, and so it defines a differential 2-form on the entire smooth variety $\X(A)$.

\subsection{Finite Type: $A_3$} Let 
\[ A = \gmat{
0 & 1 & 0 \\
-1 & 0 & 1 \\
0 & -1 & 0 }\]
The corresponding cluster algebra $\A(A)$ is the \emph{cluster algebra of type $A_3$}.  Its nine cluster variables correspond to internal diagonals of a hexagon, and are labelled $x_{i,j}$, with $i,j\in \mathbb{Z}/6$.  Clusters correspond to complete triangulations.

The cluster variety $\X(A)$ has a unique deep point $p$, whose coordinates are given by $x_{i,i+2}=0$ and $x_{i,i+3}=-1$, for all $i$.  This deep point has Zariski tangent dimension 4, and since $\X(A)$ is 3-dimensional, we see that $p$ is a singular point.

The Weil-Petersson form can be written
\[\omega = \frac{dx_{13}\wedge dx_{14}}{x_{13}x_{14}} + \frac{dx_{14}\wedge dx_{15}}{x_{14}x_{15}} = \frac{dx_{13}\wedge dx_{24}}{x_{14}}+\frac{dx_{46}\wedge dx_{15}}{x_{14}}\]
The first expression is $\omega_{\mathbf{f}}$ for the cluster $\mathbf{f}=\{x_{13},x_{14},x_{15}\}$, and the second expression is clearly regular at $p$, since $x_{14}(p)\neq0$.  Therefore, despite the non-smoothness of $p$, the Weil-Petersson form still extends as a K\"ahler 2-form on all of $\X(A)$.

\subsection{Affine Type: $\widetilde{A}(1,1)$.}

Let 
\[ A=\gmat{0 & 2 \\ -2 & 0 }\]
This cluster algebra has an infinite number of cluster variables, which are denoted by $x_i$, $i\in \mathbb{Z}$, with clusters $\{x_i,x_{i+1}\}$, $i\in \mathbb{Z}$.  The cluster algebra may be finitely-presented, as
\[ \mathbb{C}[x_0,x_1,x_2,x_3]/\langle x_0x_2-(x^2_1+1),x_1x_3-(x_2^2+1)\rangle\]

The cluster variety $\X(A)$ has four deep points, denoted $p_j$, $j\in \mathbb{Z}/4$, given by
\[ x_k(p_j)=\left\{\begin{array}{ll}
\sqrt{-1} & \text{if } k\equiv j\,\mod 4 \\
-\sqrt{-1} & \text{if } k\equiv j+2 \,\mod 4 \\
0 & \text{if } k\equiv j+1 \,\mod 2 \end{array}\right\}\]

The Weil-Petersson form is then
\[ \omega = 2\frac{dx_0\wedge dx_1}{x_0x_1} = \frac{dx_0\wedge dx_2}{x_1^2} = \frac{dx_{-1}\wedge dx_1}{x_0^2}\]
The second expression is clearly regular at $p_i$ when $i$ is odd, and the third expression is clearly regular at $p_i$ when $i$ is even.  A globally-regular expression for $\omega$ is
\[ x_0x_3dx_1\wedge dx_2 - \frac{x_1x_3}{2}dx_0\wedge dx_2 - \frac{x_0x_2}{2}dx_1\wedge dx_3+x_1x_2 dx_1\wedge dx_2\]
We do not know of a general method for producing these expressions, though they exist for acyclic cluster algebras by Theorem \ref{thm: Main}.

\subsection{The Markov cluster algebra.}

Let
\[A = \gmat{
0 & 2 & -2\\
-2 & 0 & 2 \\
2 & -2 & 0} \]
This defines the \textbf{Markov cluster algebra}, which is (in)famous both for its number theoretic applications and its pathological behavior (ie, it is infinitely generated \cite[Theorem 1.24]{BFZ05}).  It is the simplest example of a non-acyclic cluster algebra, and so the results of this paper do not apply.

The cluster algebra $\A(A)$ is $\mathbb{N}$-graded, by letting every cluster variable have degree 1.  The K\"ahler 2-forms are an $\mathbb{N}$-graded $\A(A)$-module, with $\deg(da) = \deg(a)-1$.  This extends to a $\mathbb{Z}$-grading on the rational K\"ahler 2-forms.  Then it is clear that the Weil-Petersson form has degree -2, and thus it is \underline{not} a regular K\"ahler 2-form.

What does this mean geometrically?  The cluster variety $\X(A)$ has a deep point $p_0$, where every cluster variable is zero.\footnote{This is not the only deep point; there are many others.}  The above observation shows that the equations defining $\omega$ blow up at $p_0$.  However, Lemma \ref{lemma: Main} can be used to show $\omega$ is regular everywhere else on $\X(A)$.

\bibliography{MyNewBib}{}

\def\cprime{$'$} \def\cprime{$'$} \def\cprime{$'$} \def\cprime{$'$}
  \def\cprime{$'$}
\providecommand{\bysame}{\leavevmode\hbox to3em{\hrulefill}\thinspace}
\providecommand{\MR}{\relax\ifhmode\unskip\space\fi MR }
\providecommand{\MRhref}[2]{%
  \href{http://www.ams.org/mathscinet-getitem?mr=#1}{#2}
}
\providecommand{\href}[2]{#2}
\begin{thebibliography}{GSV10}

\bibitem[BFZ05]{BFZ05}
Arkady Berenstein, Sergey Fomin, and Andrei Zelevinsky, \emph{Cluster algebras.
  {III}. {U}pper bounds and double {B}ruhat cells}, Duke Math. J. \textbf{126}
  (2005), no.~1, 1--52. \MR{2110627 (2005i:16065)}

\bibitem[FG06]{FG06}
Vladimir Fock and Alexander Goncharov, \emph{Moduli spaces of local systems and
  higher {T}eichm\"uller theory}, Publ. Math. Inst. Hautes \'Etudes Sci.
  (2006), no.~103, 1--211. \MR{2233852 (2009k:32011)}

\bibitem[FZ02]{FZ02}
Sergey Fomin and Andrei Zelevinsky, \emph{Cluster algebras. {I}.
  {F}oundations}, J. Amer. Math. Soc. \textbf{15} (2002), no.~2, 497--529
  (electronic). \MR{1887642 (2003f:16050)}

\bibitem[GSV05]{GSV05}
Michael Gekhtman, Michael Shapiro, and Alek Vainshtein, \emph{Cluster algebras
  and {W}eil-{P}etersson forms}, Duke Math. J. \textbf{127} (2005), no.~2,
  291--311. \MR{2130414 (2006d:53103)}

\bibitem[GSV10]{GSV10}
\bysame, \emph{Cluster algebras and {P}oisson geometry}, Mathematical Surveys
  and Monographs, vol. 167, American Mathematical Society, Providence, RI,
  2010. \MR{2683456}

\end{thebibliography}
\bibliographystyle{amsalpha}

\end{document}